\documentclass[12pt]{amsart} 

\usepackage{fullpage}
\usepackage{url,amssymb,amsmath,amsfonts,amsthm,mathrsfs}

\usepackage[all]{xy} \SelectTips{cm}{10}

\usepackage{lipsum} 

\usepackage{color}

\definecolor{webcolor}{rgb}{0.8,0,0.2}
\definecolor{webbrown}{rgb}{.6,0,0}
\usepackage[
        colorlinks,
        linkcolor=webbrown,  filecolor=webcolor,  citecolor=webbrown, 
        backref,
        pdfauthor={David Zywina}, 
       pdftitle={An elliptic surface with infinitely many fibers for which the rank does not jump},
]{hyperref}
\usepackage[alphabetic,backrefs,lite]{amsrefs} % for bibliography

\numberwithin{equation}{section}

\newcommand{\QQ}{\mathbb Q}
\newcommand{\RR}{\mathbb R}

\newcommand{\ZZ}{\mathbb Z}

\newcommand{\calJ}{\mathcal J}

\newcommand{\calA}{\mathcal A}

\newcommand{\calN}{\mathcal N}

\def\Sel{\operatorname{Sel}} 

\def\Gal{\operatorname{Gal}}
\def\ord{\operatorname{ord}}

\newcommand{\legendre}[2]{\genfrac{(}{)}{}{}{#1}{#2}}

\newcommand{\defi}[1]{\textsf{#1}} % for defined terms

\newcommand\blank[1]{}

\def\bbar#1{\setbox0=\hbox{$#1$}\dimen0=.2\ht0 \kern\dimen0 
\overline{\kern-\dimen0 #1}}
\newcommand{\Qbar}{{\overline{\mathbb Q}}}

\newtheorem{thm}{Theorem}[section]
\newtheorem{lemma}[thm]{Lemma}

\newtheorem{conj}[thm]{Conjecture}

\theoremstyle{definition}

\theoremstyle{remark}
\newtheorem{remark}[thm]{Remark}

\newenvironment{romanenum}{\hfill \begin{enumerate} }{\end{enumerate}}

\begin{document}

\title{An elliptic surface with infinitely many fibers for which the rank does not jump}
\subjclass[2020]{Primary 11G18; Secondary 14J27}

% 11G05 Elliptic curves over global fields
% 14J27 Elliptic surfaces, elliptic or Calabi-Yau fibrations
% \keywords{}
\author{David Zywina}
\address{Department of Mathematics, Cornell University, Ithaca, NY 14853, USA}
\email{zywina@math.cornell.edu}

\begin{abstract}
Let $E$ be a nonisotrivial elliptic curve over $\QQ(T)$ and denote the rank of the abelian group $E(\QQ(T))$ by $r$.   For all but finitely many $t\in \QQ$, specialization will give an elliptic curve $E_t$ over $\QQ$ for which the abelian group $E_t(\QQ)$ has rank at least $r$.   Conjecturally, the set of $t\in \QQ$ for which $E_t(\QQ)$ has rank exactly $r$ has positive density.   We  produce the first known example for which $E_t(\QQ)$ has rank $r$ for infinitely many $t\in \QQ$.     For our particular $E/\QQ(T)$ which has rank $0$, we will make use of a theorem of Green on $3$-term arithmetic progressions in the primes to produce $t\in \QQ$ for which $E_t$ has only a few bad primes that we understand well enough to perform a $2$-descent.

\end{abstract}

\maketitle

\section{Introduction}

Let $E$ be an elliptic curve over the function field $\QQ(T)$ that is nonisotrivial, i.e., its $j$-invariant does not lie in $\QQ$.   Fix a Weierstrass model of $E$ with coefficients in $\QQ[T]$ and denote its discriminant by $\Delta$. For all $t\in \QQ$ with $\Delta(t)\neq 0$, evaluating the coefficients of the model by $t$ gives an elliptic curve $E_t$ over $\QQ$.

The group $E(\QQ(T))$ is a finitely generated abelian group whose rank we will denote by $r$.   A theorem of Silverman \cite{MR703488} says that the group $E_t(\QQ)$ has rank \emph{at least} $r$ for all but finitely many $t\in \QQ$.   Let $\calN(E)$ and $\calJ(E)$ be the set of $t\in \QQ$ with $\Delta(t)\neq 0$ for which $E_t(\QQ)$ has rank equal to $r$ and rank strictly greater than $r$, respectively.   

Conjecturally the sets $\calN(E)$ and $\calJ(E)$ both have positive density in $\QQ$ with respect to the natural height, cf.~\cite[\S4]{MR4640089} for a heuristic.   There has been much study on the set $\calJ(E)$ which describes the $E_t$ for which their rank ``jumps'', cf.~\cite{MR3007150} and the references therein.  We will instead focus on the set $\calN(E)$ and the following weaker conjecture.

\begin{conj} \label{C:main}
The set $\calN(E)$ is infinite, i.e., there are infinitely many $t\in \QQ$ for which $E_t(\QQ)$ has rank $r$.
\end{conj}

Our main result gives the first unconditional example for which Conjecture~\ref{C:main} holds.

\begin{thm} \label{T:main}
Let $E/\QQ(T)$ be the elliptic curve defined by the equation $y^2=x(x^2-x+T)$.  The group $E(\QQ(T))$ has rank $0$ and $E_t(\QQ)$ has rank $0$ for infinitely many $t\in \QQ$.
\end{thm}

Take $E/\QQ(T)$ as in Theorem~\ref{T:main}.  The goal is to find specializations $E_t/\QQ$ for which the curve has few bad primes and for which they are all explicitly understood.  In order to bound the rank of $E_t(\QQ)$, we will bound the cardinality of its $2$-Selmer group and this will depend on the knowledge of these bad primes.

Let us describe the specializations we use in our proof of Theorem~\ref{T:main}.   Take any positive integers $m$ and $n$ for which $m$, $m+n$ and $m+2n$ are all primes that are congruent to $3$ modulo $8$.  With $t:=\tfrac{m+n}{2m} \in \QQ$,  we shall prove that $E_t(\QQ)$ has rank $0$.   The elliptic curve $E_t$ has good reduction away from the primes $2$, $m$, $m+n$ and $m+2n$. 

A theorem of Green \cite{MR2180408}, later generalized by Green and Tao \cite{MR2415379}, will be used to show that there are infinitely many such arithmetic progressions of primes; this is the source of the infiniteness in Theorem~\ref{T:main}.   Alternatively, this could be proved with a minor modification of the classical circle method argument that van der Corput used  in 1939 to prove that there are infinitely many $3$-term arithmetic progressions of primes.
% Aside: I have not worked out this last statement, but Ben Green tells me it is true!

For our elliptic curve $E$, it is easy to show that the set $\calJ(E)$ is also infinite.   Indeed, using Silverman's result one can prove that $(1,b)$ is a point of infinite order on $E_{b^2}$ for all but finitely many $b\in \QQ$.

In a followup paper \cite{Zyw25}, we will give another example of Conjecture~\ref{C:main} with $r=2$.  

\subsection{Some earlier conditional results} \label{SS:earlier}

Let $E/\QQ(T)$ be the elliptic curve given by $y^2=x(x+1)(x+T)$.  Caro and Pasten \cite{MR4640089} showed that $E$ satisfies Conjecture~\ref{C:main} if there are infinitely many Mersenne primes.  Moreover, given any Mersenne prime $p=2^q-1$ with $q\geq 5$, they show that $E_{2^q}(\QQ)$ has rank $0$.  Note that such an elliptic curve $E_{2^q}$ has good reduction away from $2$ and $p$.    The existence of infinitely many Mersenne primes is of course a famous open problem.

Let $E/\QQ(T)$ be the elliptic curve given by $y^2=x^3-(T+1)/4\cdot x^2-x$.   If $p$ is a prime of the form $t^2+64$ for an integer $t$, then one can show that $E_t(\QQ)$ has rank $0$.     These elliptic curves have been studied in \cite{MR337999,MR371904,MR2052021}.   Note that such an elliptic curve $E_{t}$ has good reduction away from $p$.   The existence of infinitely many primes of the form $t^2+64$ with $t\in \ZZ$ is an open problem (a special case of the Bunyakovsky conjecture).

\subsection{Aside: the isotrivial case}

For Conjecture~\ref{C:main}, it is important that $E$ is assumed to be nonisotrivial and not just nonconstant.  Consider the \emph{isotrivial} elliptic curve $E/\QQ(T)$ defined by $y^2=x(x^2-(7+7T^4))^2$.  Cassels and Schinzel \cite{MR663485} observed that $E(\QQ(T))$ has rank $0$ and expected that $E_t(\QQ)$ has rank at least $1$ for all $t\in \QQ$.   Indeed, the root number of each $E_t$ is $-1$ and hence the rank of $E_t(\QQ)$ should be odd by the parity conjecture.

It is straightforward to find isotrivial and nonconstant examples for which the conclusion of Conjecture~\ref{C:main} holds.   Consider the elliptic curve $E/\QQ(T)$ defined by the equation $y^2=x^3+Tx$.  Then $E_p(\QQ)$ has rank $0$ for all primes $p$ that are congruent to $7$ or $11$ modulo $16$, cf.~\cite[Proposition~6.2]{Silverman}.

What makes the nonisotrivial case more difficult is that it is harder to produce $t\in \QQ$ for which $E_t$ has bad reduction at only a few primes which are easy to describe.  This is clear from our example and the earlier conditional examples in \S\ref{SS:earlier}.

\section{Main computation} \label{S:rank 0}

Consider any positive integers $m$ and $n$ for which $m$, $m+n$ and $m+2n$ are all primes that are congruent to $3$ modulo $8$.   Set $a:=-4m^2$ and $b:=8m^3(m+n)$, and define the elliptic curve $E$ over $\QQ$ by
\begin{align} \label{E:new uv model}
y^2=x(x^2+a x + b ) = x(x^2-4m^2 x+8m^3(m+n)). 
\end{align}
In this section we shall prove that $E(\QQ)$ has rank $0$. 

\begin{remark}
Set $t:=(m+n)/(2m) \in \QQ$.  In our proof of Theorem~\ref{T:main} in \S\ref{S:proof}, we will see that this curve is isomorphic to the elliptic curve $E_t/\QQ$ with notation as in Theorem~\ref{T:main}.
\end{remark}

Set $a':=-2a=8m^2$ and $b':=a^2-4b=-16 m^3(m+2n)$, and define the elliptic curve $E'$ over $\QQ$ by 
\begin{align} \label{E:new uv model 2}
y^2=x(x^2+a'x+b')=x(x^2+8m^2x-16m^3(m+2n)).
\end{align}
There is an isogeny $\phi\colon E\to E'$ given by $\phi(x,y)=(y^2/x^2, y(b-x^2)/x^2)$ whose kernel $E[\phi]$ is cyclic of order $2$ and generated by $(0,0)$.   Let $\hat\phi\colon E'\to E$ be the dual isogeny of $\phi$; its kernel $E'[\hat\phi]$ is generated by the $2$-torsion point $(0,0)$ of $E'$.

The discriminant of the Weierstrass models (\ref{E:new uv model}) is $-2^{14} m^9 (m+n)^2 (m+2n)$.  Therefore, $E$ and $E'$ both have good reduction at all primes away from the set $\{2,m,m+n,m+2n\}$.

 For each prime $p$, we let $c_p(E)$ and $c_p(E')$ be the Tamagawa number of $E$ and $E'$, respectively, at $p$.  For each prime $p$, we will denote by $\ord_p$ the discrete valuation on $\QQ_p$ with valuation ring $\ZZ_p$ normalized so that $\ord_p(p)=1$.
 
 Let $W(E)$ be the global root number of $E/\QQ$.   We will now show that $W(E)=1$; the Birch and Swinnerton--Dyer conjecture would imply that this is a necessary condition for $E(\QQ)$ to have rank $0$. 

\begin{lemma}
\label{L:Tate's algorithm 1}
\begin{romanenum}
\item \label{L:Tate's algorithm 1 i}
We have $W(E)=1$.
\item \label{L:Tate's algorithm 1 ii}
We have $\prod_{p} c_p(E)=8$.
\end{romanenum}
\end{lemma}
\begin{proof}
The root number $W(E)$ is the product of the local root numbers $W_v(E)$ over the places $v$ of $\QQ$, see~\cite{MR1219633} for descriptions of the local root numbers.    The local root number at the archimedean place is $-1$ and $W_p(E)=1$ for all primes $p$ for which $E$ has good reduction.   So to determine $W(E)$, we need only compute $W_p(E)$ with $p\in \{2,m,m+n,m+2n\}$.

The elliptic curve $E/\QQ$ is given by the Weierstrass equation
\[
y^2=x(x^2-4m^2 x+8m^3(m+n)) = x( (x-2m^2)^2 + 4m^3(m+2n))
\]
which has discriminant $\Delta=-2^{14} m^9 (m+n)^2 (m+2n)$.   We will make use of Tate's algorithm \cite[Algorithm 9.4]{SilvermanII} at each bad prime.  In particular, we will find that the above Weierstrass model is minimal.

First consider the prime $p:=m$.  Applying Tate's algorithm, we find that $E$ has Kodaira symbol $\operatorname{III}^*$ at $p$ and hence $c_p(E)=2$.   If $p>3$, then \cite[Proposition 2(v)]{MR1219633} implies that $W_p(E)=\legendre{-2}{p}=1$, where the last equality uses that $p\equiv 3 \pmod{8}$.   When $p=3$, we also have $W_p(E)=1$; this can be read off \cite[Table 2]{MR1647190} by using only the Kodaira symbol.

Consider the prime $p:=m+n$.   We have $\ord_p(\Delta)=2$ and $y^2 \equiv -4m^2 \cdot x^2 + x^3 \pmod{p}$, so $E$ has Kodaira symbol $\operatorname{I}_2$ at $p$ and hence $c_p(E)=2$.   The curve $E$ has nonsplit multiplicative reduction at $p$ since $\legendre{-4m^2}{p}=\legendre{-1}{p}=-1$, where the last equality uses that $p\equiv 3 \pmod{4}$.  We have $W_p(E)=1$ by \cite[Proposition~3]{MR1219633}.  

Consider the prime $p:=m+2n$.   We have $\ord_p(\Delta)=1$ and 
\[
y^2 \equiv x(x-2m^2)^2 \equiv 2m^2\cdot (x-2m^2)^2+ (x-2m^2)^3\pmod{p},
\] 
so $E$ has Kodaira symbol $\operatorname{I}_1$ at $p$ and hence $c_p(E)=1$.   The curve $E$ has nonsplit multiplicative reduction at $p$ since $\legendre{2m^2}{p}=\legendre{2}{p}=-1$, where the last equality uses that $p\equiv 3 \pmod{8}$.  We have $W_p(E)=1$ by \cite[Proposition~3]{MR1219633}.  

Finally consider the prime $p=2$.  Applying Tate's algorithm, we find that $E$ has Kodaira symbol $\operatorname{III}^*$ at $2$ and hence $c_2(E)=2$.   The root number $W_2(E)$ can be computed using Table~1 of \cite{MR1647190} (in the notation of the table, we have $\ord_2(c_4)=7$, $\ord_2(c_6)=10$, $\ord_2(\Delta)=14$, $c_4'=-m^4 - 3m^3n\equiv 7 \pmod{8}$ and $c_6'=-5m^6 - 9m^5n\equiv 3 \pmod{8}$).  We have $W_2(E)=-1$.

We have $W(E) = - \prod_{p} W_p(E)$ and hence $W(E)=-(-1)=1$ by the above computations.  Since $c_p(E)=1$ for each prime $p$ for which $E$ has good reduction, the above computations show that $\prod_p c_p(E)=8$.    
\end{proof}

\begin{lemma} \label{L:Tate's algorithm 2}
We have $\prod_p c_p(E')=4$.
\end{lemma}
\begin{proof}
The elliptic curve $E'/\QQ$ is isomorphic to the curve given by the Weierstrass equation
\[
y^2=x(x^2+2m^2x-m^3(m+2n)) = x((x+m^2)^2-2m^3(m+n))
\]
which has discriminant $\Delta'=2^{7} m^9 (m+n) (m+2n)^2$ (replacing $x$ and $y$ in (\ref{E:new uv model 2}) by $4x$ and $8y$ will produce the above model).   Using that $m$, $m+n$ and $m+2n$ are distinct odd primes, we can apply Tate's algorithm \cite[Algorithm 9.4]{SilvermanII} for the primes $p\in \{2,m,m+n,m+2n\}$ to show that the above Weierstrass model is minimal and that the Kodaira symbols of $E$ at $2$, $m$, $m+n$ and $m+2n$ are equal to $\operatorname{II}$, $\operatorname{III}^*$, $\operatorname{I}_1$ and $\operatorname{I}_2$, respectively.     In these cases, the Tamagawa numbers are determined by the Kodaira symbols and we have $c_2(E')=1$, $c_m(E')=2$, $c_{m+n}(E')=1$ and $c_{m+2n}(E')=2$, cf.~\cite[Algorithm 9.4]{SilvermanII}.  The lemma follows since $c_p(E')=1$ for all primes $p$ for which $E'$ has good reduction.
\end{proof}

We will now compute the Selmer groups associated to the isogenies $\phi$ and $\hat\phi$.   For basic definitions and results see \cite[\S X.4]{Silverman}.  In particular, \cite[\S X.4 Example 4.8]{Silverman} contains the relevant formulae for our computations. Set $\Gal_\QQ:=\Gal(\Qbar/\QQ)$.  Starting with the short exact sequence $0\to E[\phi]\to E \xrightarrow{\phi} E'\to 0$ and taking Galois cohomology yields an exact sequence
\[
0 \to E(\QQ)[\phi] \to E(\QQ)\xrightarrow{\phi} E'(\QQ) \xrightarrow{\delta} H^1(\Gal_\QQ, E[\phi]).
\]
The image of $\delta$ lies in the $\phi$-\defi{Selmer group} $\Sel_\phi(E/\QQ) \subseteq H^1(\Gal_\QQ,E[\phi])$.   Since $E[\phi]$ and $\{\pm 1\}$ are isomorphism $\Gal_\QQ$-modules, we have isomorphisms 
\begin{align} \label{E:H1 isom}
H^1(\Gal_\QQ, E[\phi])\xrightarrow{\sim} H^1(\Gal_\QQ, \{\pm 1\}) \xrightarrow{\sim} \QQ^\times/(\QQ^\times)^2.
\end{align}
Using (\ref{E:H1 isom}) as an identification, we may view $\delta$ as a homomorphism $E'(\QQ)\to  \QQ^\times/(\QQ^\times)^2$.   For any point $(x,y)\in E'(\QQ)-\{0,(0,0)\}$, we have $\delta((x,y))=x\cdot (\QQ^\times)^2$.  We also have $\delta(0)=1$ and $\delta((0,0))=b'\cdot (\QQ^\times)^2$.

For each $d \in \QQ^\times$, let $C_d$ be the smooth projective curve over $\QQ$ defined by the affine equation
\[
dw^2=d^2+a' d z^2 + b' z^4.
\]
Using (\ref{E:H1 isom}), we can identify $\Sel_\phi(E/\QQ)$ with a subgroup of $ \QQ^\times/(\QQ^\times)^2$.  In fact, we have
\[
\Sel_\phi(E/\QQ)=\{ d\in \QQ^\times/(\QQ^\times)^2: C_d(\QQ_v)\neq \emptyset \text{ for all places $v$ of $\QQ$}\}.
\]

\begin{lemma} \label{L:Selmer 1}
We have $|\Sel_\phi(E/\QQ)|=2$.
\end{lemma}
\begin{proof}
Take any squarefree integer $d$ that represents a square class in $\Sel_\phi(E/\QQ)$.  We have $C_d(\QQ_v)\neq \emptyset$ for all places $v$ of $\QQ$.   By changing variables, we see that $C_d$ is isomorphic to the smooth projective curve $C_d'$ over $\QQ$ given by the affine model
\begin{align} \label{E:Cd'}
y^2=d x^4 + a'/4 \cdot x^2 + b'/(16d)= dx^4 +2m^2 x^2 -m^3(m+2n)/d.
\end{align}

First suppose that $d$ is divisible by a prime $p\nmid m(m+2n)$.   Since $C_d'(\QQ_p)\neq \emptyset$, there is a point $(x,y) \in \QQ_p^2$ satisfying (\ref{E:Cd'}); the points at infinity are not defined over $\QQ_p$ since $d$ is not a square in $\QQ_p$.   If $x\in \ZZ_p$, then from (\ref{E:Cd'}) we find that $\ord_p(y^2)$ is equal to $\ord_p(-m^3(m+2n)/d)=-1$.  If $x\notin \ZZ_p$, then from (\ref{E:Cd'}) we find that $\ord_p(y^2)$ is equal to $\ord_p(dx^4)=1+4\ord_p(x)$.   In either case, $\ord_p(y^2)=2\ord_p(y)$ is an odd integer which is a contradiction.   Therefore, if a prime divides $d$, then it must be $m$ or $m+2n$.    In particular, $d\in \{\pm 1, \pm m,  \pm (m+2n), \pm m(m+2n) \}$.

Now suppose that $d\equiv \pm 3 \pmod{8}$.   The integer $d$ is not a square in $\QQ_2$, so the points at infinity of the model (\ref{E:Cd'}) are not defined over $\QQ_2$.   Since $C_d'(\QQ_2)\neq \emptyset$, there is a point $(x,y) \in \QQ_2^2$ satisfying (\ref{E:Cd'}).   First suppose that $x\in \ZZ_2$ and hence $y\in \ZZ_2$ as well.  If $x\in 2\ZZ_2$, then $y^2 \equiv -m^3(m+2n)/d \equiv \pm 3 \pmod{8}$.   If $x\in \ZZ_2^\times$, then  $y^2\equiv d  + 2  - 1/d \equiv d+ 2 -d\equiv 2 \pmod{8}$.  In both of these computations we have used that $m$ and $m+2n$ are congruent to $3$ modulo $8$.  Since $3$, $-3$ and $2$ are not squares modulo $8$, we deduce that $x\notin \ZZ_2$.   Define $e:=-\ord_2(x)\geq 1$.   Since $m$ and $d$ are odd, we find that $2\ord_2(y)=\ord_2(y^2)=\ord_2(dx^4)=-4e$ and hence $\ord_2(y)=-2e$.  Multiplying (\ref{E:Cd'}) by $2^{4e}$ gives $(2^{2e}y)^2=d(2^ex)^4+ 2^{2e+1}m^2(2^ex)^2- 2^{4e} m^3(m+2n)/d$.  Reducing modulo $8$, we find that $d$ is a square modulo $8$ which contradicts that $d\equiv \pm 3 \pmod{8}$.  

We thus have $d\not \equiv \pm 3 \pmod{8}$.  Since $m$ and $m+2n$ are congruent to $3$ modulo $8$, we must have $d\in \{\pm 1,\pm m(m+2n)\}$.

Now suppose that $d=-1$.  The curve $C_d'$ is given by the model
\begin{align} \label{E:Cd' -1}
y^2= -x^4 +2m^2 x^2 +m^3(m+2n)= -(x^2-m^2)^2 + 2m^3(m+n).
\end{align}
Set $p:=m+n$.  We note that $-1$ is not a square modulo $p$ since $p\equiv 3\pmod{4}$.   The integer $-1$ is not a square in $\QQ_p$ so the points at infinity of the model of $C_d'$ are not defined over $\QQ_p$.  Since $C_d'(\QQ_p)\neq \emptyset$, there is a point $(x,y) \in \QQ_p^2$ satisfying (\ref{E:Cd' -1}).   Define $z:=x^2-m^2\in \QQ_p$; we have $y^2=-z^2+2m^3p$.   If $z\in p\ZZ_p$, then $2\ord_p(y)=\ord_p(2m^3p)=1$ which is impossible.  If $z\in \ZZ_p^\times$, then $y^2\equiv -z^2 \pmod{p}$ and hence $-1$ is a square modulo $p$ which is impossible.   Define $e:=-\ord_p(z)\geq 1$.   We have $2\ord_p(y)=\ord_p(y^2)=\ord_p(z^2)=-2e$ and hence $\ord_p(y)=-e$.  Therefore, $(p^ey)^2=-(p^ez)^2+2m^3p^{1+2e}$ and reducing modulo $p$ shows that $-1$ is a square modulo $p$ which is impossible.    Therefore, $d\neq -1$.

We have now shown that every element of $\Sel_\phi(E/\QQ)$ is represented by the square class of an integer $d\in \{1,\pm m(m+2n) \}$.  Since $\Sel_\phi(E/\QQ)$ is an abelian $2$-group, it must be cyclic of order $1$ or $2$.  The group $\Sel_\phi(E/\QQ)$ has order $2$ since it contains $\delta((0,0))=b'\cdot (\QQ^\times)^2=-m(m+2n)\cdot (\QQ^\times)^2$ and $m(m+2n)$ is not a square.
\end{proof}

We now compute the cardinality of the Selmer group $\Sel_{\hat{\phi}}(E'/\QQ)$.

\begin{lemma} \label{L:Selmer 2}
We have $|\Sel_{\hat\phi}(E'/\QQ)|= 2$.
\end{lemma}
\begin{proof}
For a choice of minimal Weierstrass model $y^2+a_1xy+a_3y=x^3+a_2x^2+a_4x+a_6$ of $E/\QQ$, we define the invariant differential $\omega:=dx/(2y+a_1x+a_3)$ on $E$.  We denote the integral of $|\omega|$ over $E(\RR)$ by $\Omega_E$.  We similarly define a differential $\omega'$ on $E'$ and a period $\Omega_{E'}$.  

By equation (6.2) of \cite{MR2021618}, which is a reformulation of a result of Cassels from \cite{MR179169}, we have
\[
\frac{|\Sel_{\hat\phi}(E'/\QQ)|}{|\Sel_{\phi}(E/\QQ)|} =
\frac{|E'(\QQ)[\hat{\phi}]|}{|E(\QQ)[\phi]|} \cdot \frac{\Omega_E}{\Omega_{E'}} \cdot \prod_{p} \frac{c_p(E)}{c_p(E')}.
\]
We have $|\Sel_\phi(E/\QQ)|=2$ by Lemma~\ref{L:Selmer 1} and $\prod_p c_p(E)/c_p(E')=2$ by Lemmas~\ref{L:Tate's algorithm 1}(\ref{L:Tate's algorithm 1 ii}) and \ref{L:Tate's algorithm 2}.   Therefore,
\[
|\Sel_{\hat\phi}(E'/\QQ)| = 4 \cdot \Omega_E/\Omega_{E'}.
\]

There is a unique real number $c$ for which $c \cdot \phi^*\omega' = \omega$.   From \cite[Theorem~1.2]{MR3324930}, we have $\Omega_E/\Omega_{E'}=|c|$.  As noted in the proof of \cite[Theorem~8.2]{MR3324930}, we have $|c| \in \{1,1/2\}$.   Therefore, $\Omega_E/\Omega_{E'}$ is either $1$ or $1/2$.

Suppose that $\Omega_E/\Omega_{E'}=1$.   Since $\prod_p c_p(E)/c_p(E')=2$, \cite[Theorem~8.2]{MR3324930} implies that the order of vanishing of the $L$-function $L(E,s)$ at $s=1$ is odd. Equivalently, the global root number $W(E)$ is $-1$ which contradicts Lemma~\ref{L:Tate's algorithm 1}(\ref{L:Tate's algorithm 1 i}).  Therefore, $\Omega_E/\Omega_{E'}=1/2$ and we conclude that $|\Sel_{\hat\phi}(E'/\QQ)|=2$.
\end{proof}

We can now bound the cardinality of the $2$-Selmer group of $E/\QQ$.

\begin{lemma} \label{L:Selmer 3}
We have $|\Sel_2(E/\QQ)|\leq 2$.
\end{lemma}
\begin{proof}
By \cite[Lemma 6.1]{MR2021618}, we have an exact sequence
\[
0\to E'(\QQ)[\hat\phi]/\phi(E(\QQ)[2]) \xrightarrow{\alpha} \Sel_\phi(E/\QQ) \xrightarrow{\beta} \Sel_2(E/\QQ) \xrightarrow{\gamma} \Sel_{\hat\phi}(E'/\QQ)
\]
of groups.  The discriminant of $x^2-4m^2 x+8m^3(m+n)$ is divisible by the prime $m+2n$ exactly once and hence is not a square.  Therefore, $E(\QQ)[2]=\langle (0,0) \rangle$ and so $E'(\QQ)[\hat\phi]/\phi(E(\QQ)[2])$ is a cyclic group of order $2$.  This implies that the injective homomorphism $\alpha$ is surjective since $|\Sel_{\phi}(E/\QQ)|=2$ by Lemma~\ref{L:Selmer 1}.  By the exactness, $\beta$ is the zero map and hence $\gamma$ is an injective homomorphism $\Sel_2(E/\QQ) \hookrightarrow \Sel_{\hat\phi}(E'/\QQ)$.  The lemma is now an immediate consequence of Lemma~\ref{L:Selmer 2}.
\end{proof}

Let $r$ be the rank of $E(\QQ)$.  Since $E(\QQ)$ has a point of order $2$, we have $|E(\QQ)/2E(\QQ)|\geq 2^{1+r}$.  There is an injective homomorphism $E(\QQ)/2E(\QQ) \hookrightarrow \Sel_2(E/\QQ)$ which implies that $E(\QQ)/2E(\QQ)$ has cardinality at most $2$ by Lemma~\ref{L:Selmer 3}.     So $2^{1+r} \leq 2$ and we conclude that $r=0$.

\section{Proof of Theorem~\ref{T:main}} \label{S:proof}

Let $\calA$ be the set of primes that are congruent to $3$ modulo $8$; it has relative density $1/4$ in the set of all primes.    A theorem of Green \cite{MR2180408} implies that $\calA$ contains infinitely many arithmetic progressions of length $3$.    

Now consider one of the infinitely many pairs $(m,n)$ of positive integers for which $m$, $m+n$ and $m+2n$ are all primes that lie in $\calA$.  Define $t:=(m+n)/(2m) \in \QQ$.  The elliptic curve $E_t/\QQ$ is given by the equation $y^2=x(x^2-x+t)$.    With $x':=4m^2 x$ and $y':=8m^3 y$, we find that $E_t$ is isomorphic to the elliptic curve over $\QQ$ given by the model
\[
y'^2=x'(x'^2-4m^2x'+8m^3(m+n)).
\]
By the computation of \S\ref{S:rank 0}, we deduce that $E_t(\QQ)$ has rank $0$.    Note that $t=(m+n)/(2m)$ is in lowest terms, so from $t$ we can recover the pair $(m,n)$.   We have thus proved that $E_t(\QQ)$ has rank $0$ for infinitely many $t\in \QQ$.

Finally let $r$ be the rank of $E(\QQ(T))$.   From Silverman \cite{MR703488}, we know that $r$ is less than or equal to the rank of $E_t(\QQ)$ for all but finitely many $t\in \QQ$.  Since we have shown that $E_t(\QQ)$ has rank $0$ for infinitely many $t\in \QQ$, we deduce that $r=0$.

\end{document}